\documentclass[12pt]{amsart}
\usepackage[colorlinks=true, pdfstartview=FitV, linkcolor=blue, citecolor=blue, urlcolor=blue]{hyperref}
\usepackage{amssymb}
\calclayout

\def\C{\mathbb {C}}
\def\lie#1{\mathfrak{ #1}}

\def\lieg{\lie g}

\def\liet{\lie t}
\def\inv{^{-1}}
\newcommand{\Aut}{\operatorname{Aut}}

\newcommand{\SL}{\operatorname{SL}}

\newcommand{\SO}{\operatorname{SO}}
\newcommand{\Orth}{\operatorname{O}}
\def\GL{\operatorname{GL}}

\def\ad{\operatorname{ad}}
\def\Ad{\operatorname{Ad}}
\def\phi{\varphi}

\numberwithin{equation}{subsection}

\newtheorem{theorem}[subsection]{Theorem}

\newtheorem{proposition}[subsection]{Proposition}
\newtheorem{corollary}[subsection]{Corollary}

\theoremstyle{definition}

\theoremstyle{remark}

\title[Linear maps preserving invariants]{\boldmath Linear maps preserving invariants} 
 \author{Gerald W. Schwarz}
\thanks{Partially supported by NSA Grant H98230-06-1-0023}
\address{Department of Mathematics\\
Brandeis University\\
Waltham, MA 02454-9110}
\email{schwarz@brandeis.edu}
\subjclass{20G20, 22E46, 22E60}
\keywords{Invariant polynomials}

\begin{document}
\begin{abstract}
Let $G\subset\GL(V)$ be a complex reductive group. Let $G'$ denote $\{\phi\in\GL(V)\mid p\circ\phi=p\text{ for all }p\in\C[V]^G\}$. We show that,  in general,   $G'=G$. In case  $G$ is the adjoint group of a simple Lie algebra $\lieg$, we show that $G'$ is an order 2 extension of $G$.  We also calculate $G'$ for all representations of $\SL_2$.
\end{abstract}

\maketitle
\section{Introduction}

Our base field is $\C$, the field of complex numbers. Let $G\subset\GL(V)$ be a reductive group.
Let $G'=\{\phi\in\GL(V)\mid p\circ \phi=p\text{ for all }p\in\C[V]^G\}$. Several authors have studied the problem of determining $G'$. If $G$ is finite, then one easily sees that $G'=G$.
Solomon \cite{Sol1,Sol2} has classified many triples consisting of reductive groups
$H\subset G$ and a $G$-module $V$ such that $\C(V)^H=\C(V)^G$ (rational invariant functions). If $G$ and $H$ are semisimple, then this is the same thing as finding triples where we have equality of the polynomial invariants: $\C[V]^H=\C[V]^G$. We show that for ``general'' faithful $G$-modules $V$ we have that $G=G'$. We also compute $G'$ for all representations of $\SL_2$.

First we study the case that $G$ is the adjoint group of a simple Lie algebra $\lieg$. Our interest in this case is due to the paper of Ra{\"\i}s \cite{Rais1} where the question of determining $G'$ is raised. The case that $\lieg=\lie{sl}_n$ was also settled by him \cite{Rais2}, where it is shown that $G'/G$ is generated by the mapping $\lie{sl}_n\ni X \mapsto X^t$ where $X^t$ denotes the transpose of $X$. In \S 2  we show that, in general, $G'/G$ is generated by the element $-\psi$ where $\psi\colon\lieg\to\lieg$ is a certain automorphism of $\lieg$ of order 2. In the case of $\lie{sl}_n$, $\psi(X)=-X^t$, so that our result reproduces that of Ra{\"\i}s. The computation of $G'$ for $\lieg$ semisimple follows easily from the case that $\lieg$ is simple.  In \S 3  we prove our result that $G=G'$ for general $G$ and general $G$-modules $V$. In \S 4 we consider representations of $\SL_2$.

We thank M.~ Ra{\"\i}s for bringing his paper \cite{Rais1} to our attention and we thank him and  D.~Wehlau for help and advice.

\section{The adjoint case}

\begin{proposition} Let $\lieg$ be a semisimple Lie algebra, so we have  $\lieg=\lieg_1\oplus\dots\oplus\lieg_r$ where the $\lieg_i$ are simple ideals.
Let $\phi\in G'$. Then $\phi(\lieg_i)=\lieg_i$ for all $i$, and $\phi|_{\lieg_i}=\pm\sigma_i$   where  $\sigma_i$ is an automorphism of $\lieg_i$.
\end{proposition} 

\begin{proof}
By a theorem of Dixmier \cite{Dix79} we know that the Lie algebra of  $G'$ is $\ad(\lieg)\subset\lie{gl}(\lieg)$. Thus $\phi$ acts on $\ad(\lieg)\simeq\lieg$ via an automorphism $\sigma$ where  $\phi\circ\ad X\circ \phi\inv=\ad\sigma(X)$ for $X\in\lieg$.  Since $\phi$ induces the identity on $\C[\lieg]^G$, so does $\sigma$, and it follows that  $\sigma=\prod_i\sigma_i$ where $\sigma_i\in\Aut(\lieg_i)$, $i=1,\dots,r$.
By Schur's lemma,  $\phi\circ\sigma\inv$ restricted to   $\lieg_i$  is multiplication by some scalar $\lambda_i\in\C^*$,  $i=1,\dots,r$. Since $\Aut(\lieg_i)$ and $G'$ preserve  the invariant of degree 2 corresponding to the Killing form on each $\lieg_i$ we must have that $\lambda_i=\pm 1$, $i=1,\dots,r$.
\end{proof}

From now on we assume   that $\lieg$ is simple. Let $\sigma\in\Aut(\lieg)$. Then we know that, up to multiplication by an element of $G=\Aut(\lieg)^0$, we can arrange that $\sigma$ preserves a fixed Cartan subalgebra $\liet\subset \lieg$. Thus we may assume that $\phi$ preserves $\liet$. Let $T$ denote the corresponding maximal torus of $G$.

\begin{corollary}
 We may modify $\phi$ by an element of $G$ so that $\phi$ is the identity on $\liet$.
 \end{corollary}
 
 \begin{proof} By Chevalley's theorem,  restriction to $\liet$ gives an isomorphism of  $\C[\lieg]^G$ with $\C[\liet]^W$ where $W$ is the Weyl group of $\lieg$. Thus the restriction of $\phi$ to $\liet$ coincides with an element of $W$, where every element of $W$ is the restriction of an element of $G$ stabilizing $\liet$. Thus we may assume that $\phi$ is the identity on $\liet$. 
   \end{proof}

Let $\Phi$ be the set of roots and  $\Phi^+$ a choice of positive roots. Let $\Pi$ denote the set of simple roots.  Since $\phi=\pm\sigma$ is the identity on $\liet$, $\sigma(x)=c_\sigma x$ for all $x\in\liet$  where  $c_\sigma=\pm 1$. Hence either $\sigma$ sends each $\lieg_\alpha$ to itself or it sends each $\lieg_\alpha$ to $\lieg_{-\alpha}$, $\alpha\in \Phi$.  Choose nonzero elements $x_\alpha\in\lieg_\alpha$, $\alpha\in\Pi$, and choose elements $y_\alpha\in\lieg_{-\alpha}$ such that $(x_\alpha,y_\alpha,[x_\alpha,y_\alpha])$ is an $\lie{sl}_2$-triple. Let $\psi$ denote the unique order 2 automorphism of $\lieg$ such that $\psi(x)=-x$, $x\in\liet$ and $\psi(x_\alpha)=-y_\alpha$, $\alpha\in\Pi$ (see \cite[14.3]{Hum72}).

\begin{proposition}\label{prop:csigma}
\begin{enumerate}
\item If $c_\sigma=1$, then $\sigma$ is inner.  
\item If $c_\sigma=-1$, then   $\sigma$ differs from $\psi$ by an element of $\Ad(T)$.
\end{enumerate}

\end{proposition}

\begin{proof}
If $c_\sigma=1$, then $\sigma(x_\alpha)=c_\alpha x_\alpha$, $c_\alpha\in\C$, $\alpha\in\Pi$. There is a $t\in T$ such that $\Ad(t)(x_\alpha)=c_\alpha x_\alpha$, $\alpha\in \Pi$. It follows that $\sigma=\Ad(t)\in G$.  If $c_\sigma=-1$, we can modify $\sigma$ by an element of $T$ so that it becomes $\psi$.
\end{proof}

   \begin{proposition} \label{prop:inner} Let $\lieg$ be simple.
 Then the following are equivalent.
  \begin{enumerate}
\item Every representation of $\lieg$ is self-dual.\label{selfdual}
\item The automorphism $\psi$ is inner. \label{inner}
\item The generators of $\C[\lieg]^G$ have even degree.\label{evendegree}
\item $\lieg$ is of  the following type: \label{type}
\begin{enumerate}
\item  $\mathsf{B_n}$, $n\geq 1$, 
\item  $\mathsf{C_n}$, $n\geq 3$,
\item  $\mathsf{D_{2n}}$, $n\geq 2$, 
\item  $\mathsf{E_7}$, $\mathsf{E_8}$, $\mathsf{F_4}$ or $\mathsf{G_2}$.
\end{enumerate}
\end{enumerate}
\end{proposition}
  
\begin{proof}
The equivalence of \eqref{selfdual}, \eqref {evendegree} and\eqref{type} is well-known. Now given a highest weight vector $\lambda$ of $\lieg$, the highest weight vector of the corresponding dual representation $V(\lambda)^*$ is $-\rho(\lambda)$ where $\rho$ is the unique element of the Weyl group $W$ which sends $\Phi^+$ to $\Phi^-$ (\cite[\S 21, Exercise 6]{Hum72}. Suppose that we have \eqref{inner}. Then, since  $\psi$ is inner and it normalizes $\liet$, it gives an element of  $W$, namely $\rho$, so that $V(\lambda)^*\simeq V(\lambda)$ for all $\lambda$ and \eqref{selfdual}  holds. Conversely, if \eqref{selfdual} holds,  then $-\rho$ is the identity on the set of weights, hence   $\rho(\alpha)=-\alpha$ for all $\alpha\in\Phi$. It follows that $\rho\circ\psi$ is an automorphism  of $\lieg$ which is the identity on $\liet$ and sends $\lieg_\alpha$ to $\lieg_\alpha$ for all $\alpha$. Then  $\rho\circ\psi\in\Ad(T)$ so that $\psi$ is inner.
\end{proof}

\begin{theorem}\label{thm:main}
The group $G'/G$ has order $2$, generated by $-\psi$. 
\end{theorem}
\begin{proof} 
If $\phi=\sigma\in\Aut(\lieg)$, then Proposition \ref{prop:csigma} shows that $\phi=\sigma\in G$. If $\phi=-\sigma$, then by Proposition \ref{prop:csigma} we may assume that $\phi=-\psi$. Now $-\psi$ induces an automorphism of $\C[\lieg]^G$ and $-\psi$ is the identity on $\liet$. Hence Chevalley's theorem shows that $-\psi\in G'$ and we know that $-\psi$ generates $G'/G$. Moreover, $-\psi$ is not in $\Aut(\lieg)$, so that $-\psi\not\in G$.
  \end{proof}
 
  \begin{corollary}\label{cor:main}
Suppose that $\psi$ is inner. Then $G'/G$ is generated by multiplication by $-1$.
\end{corollary}

We leave it to the reader to formulate versions of Theorem \ref{thm:main} and Corollary \ref{cor:main} for the semisimple case.

  \section{The general case}
  We have a  finite dimensional vector space  $V$and $G$ is a reductive subgroup of $\GL(V)$. Let $G':=\{\phi\in\GL(V)\mid p\circ\phi=p \text{ for all }p\in\C[V]^G\}$. We show that, ``in general,'' we have $G'=G$.
  
  Let $U$ denote the subset of $V$ consisting of closed $G$-orbits with trivial stabilizer.   It follows from Luna's slice theorem \cite{Luna} that $U$ is open in $V$.
  \begin{theorem}\label{thm:G'=G}
Suppose that $V\setminus U$ is of codimension $2$ in $V$. Then $G'=G$.
  \end{theorem}
  
  \begin{proof}
Let $\phi\in G'$ and let $x\in U$. Then $\phi(x)=\psi(x)\cdot x$ where $\psi\colon U\to G$ is a well-defined morphism. Since $G$ is affine, we may consider $\psi$ as a mapping from $U\to G\subset\C^n$ for some $n$ where $G$ is Zariski closed in $\C^n$. Our condition on the codimension of $V\setminus U$ guarantees that each component of $\psi$ is a regular function on $V$, hence $\psi$ extends to a morphism defined on all of $V$, with image in $G$. Now let $x\in U$. Then
$$
\phi(x)=\lim_{t\to 0}\phi(tx)/t=\lim_{t\to 0}\psi(tx)tx/t=\psi(0)(x).
$$
Thus $\phi$ is just the action of $\psi(0)\in G$, so $G'=G$.
  \end{proof}
  
  \section{Representations of  SL$_2$ }
  
  As an illustration, we consider representations of $G=\SL_2$ or $G=\SO_3$. We only consider  representations with no nonzero fixed subspace. We let $R_j$ denote the irreducible representation of dimension $j+1$, $j\geq 0$, and $kR_j$ denotes the direct sum of $k$ copies of $R_j$, $k\geq 1$. When we have a representation only containing copies of $R_j$ for $j$ even, then we are considering representations of $\SO_3$. From \cite[11.9]{Sch95} we know that all representations of $G$  satisfy the hypotheses of Theorem \ref{thm:G'=G} except for the following cases, where we compute $G'$.
\begin{enumerate}
\item For $R_1$ we have $G'=\GL_2$, for $2R_1$ we have $G'=\Orth_4$ and for $3R_1$ we have $G'=G$.
\item For $R_2$ we have $G'=\Orth_3$ and for $2R_2$ we have $G'=\Orth_3$. (Here  $G=\SO_3$.)
\item For $R_2\oplus R_1$ we have $G'=\{g'\in\GL_2\mid \det(g')=\pm 1\}$.
\item For $R_3$  the group $G'$ is the same as in case (3).
\item For $R_4$ we have $G'=G=\SO_3$.
\end{enumerate}
Most of the calculations are easy, we mention some details for some of the non obvious cases.

Suppose that our representation is $R_4$, which has generating invariants of degrees 2 and 3. The Lie algebra $\lieg'$ acts irreducibly on $R_4$, hence it is the sum of a center and a semisimple Lie algebra \cite[Ch.\ II, Theorem 11]{Jac62}. Clearly we cannot have a nontrivial center, so that $\lieg'$ is semisimple. Now  a case by case check of the possibilities forces $\lieg'=\lieg$. Suppose that $g'\in G'\setminus G$. Then conjugation by $g'$ gives an inner automorphism of $G$, hence we can correct $g'$ by an element of $G$ so that $g'$ commutes with $G$. Thus $g'$ acts on $R_4$ as a scalar. But to preserve the invariants the scalar must be 1. Thus we have $G'=G$. Similar considerations give that $\lieg'=\lieg$ in case (4), so that $G'/G$ is generated by scalar multiplication by $ i$ (since the generating invariant of $R_3$ has degree 4), which shows that $G'$ is as claimed.

In case (3), one sees that $\lieg'=\lieg$, so that generators of $G'/ G$ act  as   scalars on $R_2$ and $R_1$. Now generators of the invariants have degrees $(2,0)$ and $(1,2)$ so that $G'/G$ is generated by an element  which is multiplication by $-1$ on $R_2$ and multiplication by $i$ on $R_1$. Hence  $G'$ is as claimed.

\newcommand{\noopsort}[1]{} \newcommand{\printfirst}[2]{#1}
  \newcommand{\singleletter}[1]{#1} \newcommand{\switchargs}[2]{#2#1}
  \def\cprime{$'$}
\providecommand{\bysame}{\leavevmode\hbox to3em{\hrulefill}\thinspace}
\providecommand{\MR}{\relax\ifhmode\unskip\space\fi MR }
\providecommand{\MRhref}[2]{%
  \href{http://www.ams.org/mathscinet-getitem?mr=#1}{#2}
}\providecommand{\href}[2]{#2}

\end{document}